\documentclass[11pt]{article}
\usepackage{amsmath}
\usepackage{amssymb}
\usepackage{amsthm}
\usepackage{mathabx}
\usepackage[usenames]{color}
\usepackage{amscd}
\usepackage{dsfont}
\usepackage{indentfirst}
\usepackage{tabu}

\usepackage[colorlinks=true,linkcolor=blue,filecolor=red,
citecolor=webgreen]{hyperref}
\definecolor{webgreen}{rgb}{0,.5,0}

\def\C{{\mathds{C}}}

\def\R{{\mathbb{R}}}
\def\N{{\mathds{N}}}

\def\P{{\mathds{P}}}
\def\1{{\bf 1}}

\def\id{\operatorname{id}}

\def\pont{$\bullet$ }

\numberwithin{equation}{section}

\newtheorem{theorem}{Theorem}%[section]

\newtheorem{cor}{Corollary}

\newtheorem{prop}{Proposition}
\newtheorem{remark}{Remark}

\begin{document}

\title{{\bf Expansions of arithmetic functions of several variables with respect to certain modified unitary Ramanujan sums}}
\author{L\'aszl\'o T\'oth \\ \\ Department of Mathematics, University of P\'ecs \\
Ifj\'us\'ag \'utja 6, 7624 P\'ecs, Hungary \\ E-mail: {\tt ltoth@gamma.ttk.pte.hu}}
\date{}
\maketitle

\centerline{Bull. Math. Soc. Sci. Math. Roumanie 61 (109) No. 2, 2018, 213--223} 

\begin{abstract} We introduce new analogues of the Ramanujan sums, denoted by $\widetilde{c}_q(n)$, associated with unitary divisors, and
obtain results concerning the expansions of arithmetic functions of several variables with respect to the sums $\widetilde{c}_q(n)$. We apply
these results to certain functions associated with $\sigma^*(n)$ and $\phi^*(n)$, representing the unitary sigma function and unitary phi
function, respectively.
\end{abstract}



\section{Introduction}

Let $c_q(n)$ denote the Ramanujan sums, defined by
\begin{equation*}
c_q(n)=\sum_{\substack{1\le k\le q \\ (k,q)=1}} \exp(2\pi ikn/q),
\end{equation*}
where $q,n\in \N=\{1,2,\ldots\}$. Let $\sigma(n)$ be, as usual, the sum of divisors of $n$. Ramanujan's \cite{Ramanujan1918} classical identity
\begin{equation} \label{sigma_1}
\frac{\sigma(n)}{n}= \zeta(2) \sum_{q=1}^{\infty} \frac{c_q(n)}{q^2} \quad (n\in \N),
\end{equation}
where $\zeta$ is the Riemann zeta function, can be generalized as
\begin{equation} \label{sigma_k}
\frac{\sigma((n_1,\ldots,n_k))}{(n_1,\ldots,n_k)} = \zeta(k+1) \sum_{q_1,\ldots,q_k=1}^{\infty}
\frac{c_{q_1}(n_1)\cdots c_{q_k}(n_k)}{[q_1,\ldots,q_k]^{k+1}} \quad (n_1,\ldots,n_k\in \N),
\end{equation}
valid for any $k\in \N$. See the author \cite[Eq.\ (28)]{Tot2017}. Here $(n_1,\ldots,n_k)$ and $[n_1,\ldots,n_k]$ stand for the greatest common
divisor and the least common multiple, respectively, of $n_1,\ldots,n_k$. For $k=2$ identity \eqref{sigma_k} was deduced by Ushiroya \cite[Ex.\ 3.8]{Ush2016}.

By making use of the unitary Ramanujan sums $c^*_q(n)$, we also have
\begin{equation} \label{sigma_k_unit}
\frac{\sigma((n_1,\ldots,n_k))}{(n_1,\ldots,n_k)} = \zeta(k+1) \sum_{q_1,\ldots,q_k=1}^{\infty}
\frac{\phi_{k+1}([q_1,\ldots,q_k])}{[q_1,\ldots,q_k]^{2(k+1)}} c^*_{q_1}(n_1)\cdots c^*_{q_k}(n_k) \quad (n_1,\ldots,n_k\in \N),
\end{equation}
for any $k\in \N$. See \cite[Eq.\ (30)]{Tot2017}. The notations used here (and throughout the paper), which are not explained in the text,
are included in Section \ref{Sect_Not}. In fact, \eqref{sigma_k} and \eqref{sigma_k_unit} are special cases of the
following general result, which can be applied to several other special functions, as well.

\begin{theorem}[{\cite[Th.\ 4.3]{Tot2017}}] \label{Th_gen} Let $g:\N \to \C$ be an arithmetic function and let $k\in \N$. Assume that
\begin{equation*} \label{cond_g_multipl}
\sum_{n=1}^{\infty} 2^{k\, \omega(n)} \frac{|(\mu*g)(n)|}{n^k} < \infty.
\end{equation*}

Then for every $n_1,\ldots,n_k\in \N$,
\begin{equation*} \label{f_Raman_g_gcd}
g((n_1,\ldots,n_k)) = \sum_{q_1,\ldots,q_k=1}^{\infty} a_{q_1,\ldots,q_k} c_{q_1}(n_1) \cdots c_{q_k}(n_k),
\end{equation*}
\begin{equation*} \label{f_Raman_unit_g_gcd}
g((n_1,\ldots,n_k)) = \sum_{q_1,\ldots,q_k=1}^{\infty} a^*_{q_1,\ldots,q_k} c^*_{q_1}(n_1) \cdots c^*_{q_k}(n_k)
\end{equation*}
are absolutely convergent, where
\begin{equation*}
a_{q_1,\ldots,q_k} = \frac1{Q^k} \sum_{m=1}^{\infty} \frac{(\mu*g)(mQ)}{m^k},
\end{equation*}
\begin{equation*}
a^*_{q_1,\ldots,q_k} = \frac1{Q^k} \sum_{\substack{m=1\\ (m,Q)=1}}^{\infty} \frac{(\mu*g)(mQ)}{m^k},
\end{equation*}
with the notation $Q=[q_1,\ldots,q_k]$.
\end{theorem}

Recall that $d$ is a unitary divisor of $n$ if $d\mid n$ and $(d,n/d)=1$. Notation $d\parallel n$. Let $\sigma^*(n)$, defined as the sum
of unitary divisors of $n$, be the unitary analogue of $\sigma(n)$. Properties of the function $\sigma^*(n)$,
compared to those of $\sigma(n)$ were investigated by several authors. See, e.g., Cohen \cite{Coh1960}, McCarthy \cite{McC1986},
Sitaramachandrarao and Suryanarayana \cite{SitSur1973}, Sitaramaiah and Subbarao \cite{SitSub2007}, Trudgian \cite{Tru2015}. For example, one has
\begin{equation*}
\sum_{n\le x} \sigma^*(n) = \frac{\pi^2 x^2}{12 \zeta(3)}+ O(x(\log x)^{5/3}).
\end{equation*}

In this paper we are looking for unitary analogues of formulas \eqref{sigma_1} and \eqref{sigma_k}. Theorem \ref{Th_gen} can be applied to
the function $g(n)=\sigma^*(n)/n$. However, in this case $(\mu*g)(p)=1/p$, $(\mu*g)(p^\nu)=(1-p)/p^\nu$ for any prime $p$ and any $\nu \ge 2$.
Hence the coefficients of the corresponding expansion can not be expressed by simple special functions, and we consider the obtained
identities unsatisfactory.

Let $(k,n)_{**}$ denote the greatest common unitary divisor of $k$ and $n$. Note that $d \parallel (k,n)_{**}$ holds true if
and only if $d\parallel k$ and $d\parallel n$. Bi-unitary analogues of the Ramanujan sums may be defined as follows:
\begin{equation*}
c^{**}_q(n) = \sum_{\substack{1\le k \le q\\ (k,q)_{**}=1}} \exp(2\pi ikn/q) \quad (q,n\in \N),
\end{equation*}
but the function $q\mapsto c_q(n)$ is not multiplicative, and its properties are not parallel to the sums $c_q(n)$ and $c^*_q(n)$.
The function $c^{**}_q(q)=\phi^{**}(q)$, called bi-unitary Euler function was investigated in our paper \cite{Tot2009}.

Therefore, we introduce in Section \ref{Sect_mod_Raman_sums} new analogues of the Ramanujan sums, denoted by $\widetilde{c}_q(n)$, also
associated with unitary divisors, and show that
\begin{equation} \label{sigma_k_modified}
\frac{\sigma^*((n_1,\ldots,n_k)_{*k})}{(n_1,\ldots,n_k)_{*k}} = \zeta(k+1) \sum_{q_1,\ldots,q_k=1}^{\infty}
\frac{\phi_{k+1}([q_1,\ldots,q_k])}{[q_1,\ldots,q_k]^{2(k+1)}} \widetilde{c}_{q_1}(n_1)\cdots \widetilde{c}_{q_k}(n_k)
\quad (n_1,\ldots,n_k\in \N),
\end{equation}
where $(n_1,\ldots,n_k)_{*k}$ denotes the greatest common unitary divisor of $n_1,\ldots,n_k\in \N$. Now formulas \eqref{sigma_k},
\eqref{sigma_k_unit} and \eqref{sigma_k_modified} are of the same shape. In the case $k=1$, identity \eqref{sigma_k_modified} gives
\begin{equation*}
\frac{\sigma^*(n)}{n} = \zeta(2) \sum_{q=1}^{\infty} \frac{\phi_2(q)}{q^4} \widetilde{c}_{q}(n) \quad (n\in \N),
\end{equation*}
which may be compared to \eqref{sigma_1}.

We also deduce a general result for arbitrary arithmetic functions $f$ of several variables (Theorem \ref{Th_f_gen_modified}), which is the analogue of
\cite[Th.\ 4.1]{Tot2017}, concerning the Ramanujan sums $c_q(n)$ and their unitary analogues $c^*_q(n)$. We point out that in the case $k=1$,
Theorem \ref{Th_f_gen_modified} is the analogue of the result of Delange \cite{Del1976}, concerning classical Ramanujan sums.
As applications, we consider the functions $f(n_1,\ldots,n_k)= g((n_1,\ldots,n_k)_{*k})$, where $g$ belongs to a large class of functions
of one variable, including $\sigma^*(n)/n$ and $\phi^*(n)/n$, where $\phi^*$ is the unitary Euler function (Theorem \ref{Th_f_g}).

For background material on classical Ramanujan sums and Ramanujan expansions (Ramanujan-Fourier series) of functions of one variable
we refer to the book by Schwarz and Spilker \cite{SchSpi1994} and to the survey papers by Lucht \cite{Luc2010} and  Ram~Murty \cite{Ram2013}.
Section \ref{Sect_Prelim} includes some general properties on arithmetic functions of one and several variables defined by unitary divisors,
needed in the present paper.

\section{Premiminaries} \label{Sect_Prelim}

\subsection{Notations} \label{Sect_Not}

\pont $\P$ is the set of (positive) primes,

\pont the prime power factorization of $n\in \N$ is $n=\prod_{p\in \P}
p^{\nu_p(n)}$, the product being over the primes $p$, where all but
a finite number of the exponents $\nu_p(n)$ are zero,

\pont $(f*g)(n)=\sum_{d\mid n} f(d)g(n/d)$ is the Dirichlet convolution of the functions $f,g:\N \to \C$,

\pont $\id_s$ is the function $\id_s(n)=n^s$ ($n\in \N, s\in \R$),

\pont $\1 =\id_0$ is the constant $1$ function,

\pont $\mu$ is the M\"obius function,

\pont $\omega(n)$ stands for the number of distinct prime divisors of $n$,

\pont $\phi_s$ is the Jordan function of order $s$ given by
$\phi_s(n)=n^s\prod_{p\mid n} (1-1/p^s)$ ($s\in \R$),

\pont $\phi=\phi_1$ is Euler's totient function,

\pont $d\parallel n$ means that $d$ is a unitary divisor of $n$, i.e., $d\mid n$ and $(d,n/d)=1$
(we remark that this is in concordance with the standard notation $p^{\nu} \parallel n$ used for prime powers $p^\nu$),

\pont $(k,n)_*=\max \{d: d\mid k, d\parallel n\}$,

\pont $c^*_q(n)=\sum_{1\le k \le q, (k,q)_*=1} \exp(2\pi ikn/q)$ are the unitary Ramanujan sums ($q,n\in \N$),

\pont $(n_1,\ldots,n_k)_{*k}$ denotes the greatest common unitary divisor of $n_1,\ldots,n_k\in \N$,

\pont $(n_1,n_2)_{**}=(n_1,n_2)_{*2}$,

\pont $\sigma^*_s(n)=\sum_{d\parallel n} d^s$ ($s\in \R$),

\pont $\sigma^*(n)=\sigma^*_1(n)$ is the sum of unitary divisors of $n$,

\pont $\tau^*(n)=\sigma^*_0(n)$ is the number of unitary divisors of $n$, which equals $2^{\omega(n)}$.

\subsection{Functions defined by unitary divisors} \label{Sect_Funct_unit}

The study of arithmetic functions defined by unitary divisors goes back to Vaidyanathaswamy \cite{Vai1931} and Cohen \cite{Coh1960}.
The function $\sigma^*(n)$ was already defined above. The analog of Euler's $\phi$ function is $\phi^*$,
defined by $\phi^*(n)=\# \{k\in \N: 1\le k \le n, (k,n)_*=1\}$. The functions $\sigma^*$ and
$\phi^*$ are multiplicative and $\sigma^*(p^\nu)= p^{\nu}+1$, $\phi^*(p^\nu)=p^{\nu}-1$ for any prime powers $p^{\nu}$
($\nu \ge 1$).

The unitary convolution of the functions $f$ and $g$ is
\begin{equation*}
(f\times g)(n)=\sum_{d\parallel n} f(d)g(n/d) \quad (n\in \N),
\end{equation*}
it preserves the multiplicativity of functions, and the inverse of the constant $1$ function under the unitary convolution is $\mu^*$,
where $\mu^*(n)=(-1)^{\omega(n)}$, also multiplicative. The set ${\cal A}$ of arithmetic functions forms a unital commutative ring
with pointwise addition and the unitary convolution, having divisors of zero.

\subsection{Modified unitary Ramanujan sums} \label{Sect_mod_Raman_sums}

For $q,n \in \N$ we introduce the functions $\widetilde{c}_q(n)$ by the formula
\begin{equation} \label{sum_modified_Raman_sum}
\sum_{d\parallel q} \widetilde{c}_d(n) = \begin{cases} q, & \text{ if $q\parallel  n$,}\\ 0, & \text{ if $q\nparallel n$.}
\end{cases}
\end{equation}

It follows that $\widetilde{c}_q(n)$ is multiplicative in $q$,
\begin{equation} \label{mod_Raman_prime_pow}
\widetilde{c}_{p^\nu}(n) = \begin{cases} p^\nu-1, & \text{ if $p^\nu \parallel n$,}\\ -1, & \text{ if $p^\nu \nparallel n$,}
\end{cases}
\end{equation}
for any prime powers $p^{\nu}$ ($\nu \ge 1$) and
\begin{equation*}
\widetilde{c}_q(n) = \sum_{d\parallel (n,q)_{**}} d\mu^*(q/d) \quad (q,n\in \N).
\end{equation*}

We will need the following result.

\begin{prop} \label{Prop_mod_Raman} For any $q,n\in \N$,
\begin{equation} \label{mod_Raman_abs_id}
\sum_{d\parallel q} |\widetilde{c}_d(n)| = 2^{\omega(q/(n,q)_{**})} (n,q)_{**},
\end{equation}
\begin{equation} \label{mod_Raman_abs_ineq}
\sum_{d\parallel q} |\widetilde{c}_d(n)| \le  2^{\omega(q)} n.
\end{equation}
\end{prop}

\begin{proof} If $q=p^\nu$ ($\nu \ge 1$) is a prime power, then we have by \eqref{mod_Raman_prime_pow},
\begin{equation*}
\sum_{d\parallel p^\nu} |c^*_d(n)| = |c^*_1(n)| + |c^*_{p^\nu}(n)|= \begin{cases} 1+p^{\nu}-1 = p^\nu, & \text{ if $p^\nu \parallel n$,} \\ 1+1=2,
& \text{ otherwise.}
\end{cases}
\end{equation*}

Now \eqref{mod_Raman_abs_id} follows at once by the multiplicativity in $q$ of the involved functions, while \eqref{mod_Raman_abs_ineq} is its
immediate consequence. \end{proof}

For classical Ramanujan sums the inequality corresponding to \eqref{mod_Raman_abs_ineq} is crucial in the proof of the theorem of Delange
\cite{Del1976},  while the identity corresponding to \eqref{mod_Raman_abs_id} was pointed out by Grytczuk \cite{Gry1981}. In the case of
unitary Ramanujan sums the counterparts of \eqref{mod_Raman_abs_id} and \eqref{mod_Raman_abs_ineq} were proved by the author
\cite[Prop.\ 3.1]{Tot2017}.

\begin{prop} For any $q,n\in \N$,
\begin{equation} \label{mod_Holder_id}
\widetilde{c}_q(n) = \frac{\phi^*(q) \mu^*(q/(n,q)_{**})}{\phi^*(q/(n,q)_{**})}.
\end{equation}
\end{prop}

\begin{proof} Both sides of \eqref{mod_Holder_id} are multiplicative in $q$. If $q=p^\nu$ ($\nu \ge 1$) is a prime power, then
\begin{equation*}
\frac{\phi^*(p^\nu) \mu^*(p^{\nu}/(n,p^{\nu})_{**})} {\phi^*(p^{\nu}/(n,p^{\nu})_{**})} = \begin{cases} \frac{\phi^*(p^\nu)
\mu^*(1)}{\phi^*(1)} = p^\nu-1,
& \text{ if $p^\nu \parallel n$,} \\ \frac{\phi^*(p^\nu) \mu^*(p^{\nu})} {\phi^*(p^{\nu})}=-1,
& \text{ otherwise.} \end{cases} = c_{p^\nu}(n),
\end{equation*}
by \eqref{mod_Raman_prime_pow}. \end{proof}

For the Ramanujan sums $c_q(n)$ the identity similar to \eqref{mod_Holder_id} is usually attributed to H\"older, but was proved earlier by
Kluyver \cite{Klu1906}. In the case of the unitary Ramanujan sums $c^*_q(n)$ the counterpart of \eqref{mod_Holder_id} was deduced by
Suryanarayana \cite{Sur1970}.

Basic properties (including those mentioned above) of the classical Ramanujan sums $c_q(n)$, their unitary analogues $c^*_q(n)$ and the
modified sums $\widetilde{c}_q(n)$ can be compared by the next table.

\medskip \medskip

{\footnotesize {\tabulinesep=1.2mm
\begin{tabu} {|c|c|c|}
\hline
$\displaystyle c_q(n)= \sum_{d\mid (n,q)} d\mu(q/d)$  &
$\displaystyle c^*_q(n)= \sum_{d\mid (n,q)_*} d\mu^*(q/d)$ &
$\displaystyle \widetilde{c}_q(n)=\sum_{d\parallel (n,q)_{**}} d\mu^*(q/d)$
\\\hline
$\displaystyle c_q(n) = \frac{\phi(q) \mu(q/(n,q))}{\phi(q/(n,q)}$ &
$\displaystyle c^*_q(n) = \frac{\phi^*(q) \mu^*(q/(n,q)_{*})}{\phi^*(q/(n,q)_{*})}$ &
$\displaystyle \widetilde{c}_q(n) = \frac{\phi^*(q) \mu^*(q/(n,q)_{**})}{\phi^*(q/(n,q)_{**})}$
\\\hline
$\displaystyle c_{p^\nu}(n)= \begin{cases} p^\nu-p^{\nu-1}, &\text{if $p^\nu \mid n$,}
\\ -p^{\nu-1}, &\text{if $p^{\nu-1}\parallel n$}, \\ 0, &\text{if $p^{\nu-1}\nmid n$} \end{cases}$ &
$\displaystyle c^*_{p^\nu}(n) = \begin{cases} p^\nu-1, & \text{if $p^\nu \mid n$,} \\ -1, & \text{if $p^\nu \nmid n$} \end{cases}$ &
$\displaystyle \widetilde{c}_{p^\nu}(n) = \begin{cases} p^\nu-1, & \text{if $p^\nu \parallel n$,} \\ -1, & \text{if $p^\nu \nparallel n$} \end{cases}$
\\\hline
$\displaystyle \sum_{d\mid q} c_d(n) =  \begin{cases} q, &\text{if $q\mid n$,} \\ 0, &\text{if $q\nmid n$} \end{cases}$ &
$\displaystyle \sum_{d\parallel q} c^*_d(n) =  \begin{cases} q, &\text{if $q\mid n$,} \\ 0, &\text{if $q\nmid n$} \end{cases}$ &
$\displaystyle \sum_{d\parallel q} \widetilde{c}_d(n) =  \begin{cases} q,
& \text{if $q\parallel n$,} \\ 0, &\text{if $q\nparallel n$} \end{cases}$
\\\hline
$\displaystyle \sum_{d\mid q} |c_d(n)| = 2^{\omega(q/(n,q))}(n,q)$  &
$\displaystyle \sum_{d\parallel q} |c^*_d(n)| =  2^{\omega(q/(n,q)_*)}(n,q)_*$ &
$\displaystyle \sum_{d\parallel q} |\widetilde{c}_d(n)| = 2^{\omega(q/(n,q)_{**})}(n,q)_{**}$
\\\hline
\end{tabu}}
}

\medskip
\centerline{Table: Properties of $c_q(n)$, $c^*_q(n)$ and $\widetilde{c}_q(n)$}

\medskip

\subsection{Arithmetic functions of several variables} \label{Sect_func_several_var}

For every fixed $k\in \N$ the set ${\cal A}_k$ of arithmetic functions $f:\N^k\to \C$ of
$k$ variables is a unital commutative ring with pointwise addition and the unitary convolution defined by
\begin{equation} \label{unit_convo_sev_var}
(f\times g)(n_1,\ldots,n_k)= \sum_{d_1\parallel n_1, \ldots, d_k\parallel n_k}
f(d_1,\ldots,d_k) g(n_1/d_1, \ldots, n_k/d_k),
\end{equation}
the unity being the function $\delta_k$, where
\begin{equation*}
\delta_k(n_1,\ldots,n_k)= \begin{cases} 1, & \text{ if $n_1=\cdots =n_k=1$,}\\ 0, & \text{ otherwise.}
\end{cases}
\end{equation*}

The inverse of the constant $1$ function under \eqref{unit_convo_sev_var} is $\mu^*_k$, given by
\begin{equation*}
\mu^*_k(n_1,\ldots,n_k)=\mu^*(n_1)\cdots \mu^*(n_k)=(-1)^{\omega(n_1)+\cdots +\omega(n_k)} \quad (n_1,\ldots,n_k\in \N).
\end{equation*}

A function $f\in {\cal A}_k$ is said to be multiplicative if it is
not identically zero and
\begin{equation*}
f(m_1n_1,\ldots,m_kn_k)= f(m_1,\ldots,m_k) f(n_1,\ldots,n_k)
\end{equation*}
holds for any $m_1,\ldots,m_k,n_1,\ldots,n_k\in \N$ such that
$(m_1\cdots m_k,n_1\cdots n_k)=1$.

If $f$ is multiplicative, then it is determined by the values
$f(p^{\nu_1},\ldots,p^{\nu_k})$, where $p$ is prime and
$\nu_1,\ldots,\nu_k\in \N \cup \{0\}$. More exactly, $f(1,\ldots,1)=1$ and
for any $n_1,\ldots,n_k\in \N$,
\begin{equation*}
f(n_1,\ldots,n_k)= \prod_{p\in \P} f(p^{\nu_p(n_1)}, \ldots,p^{\nu_p(n_k)}).
\end{equation*}

Similar to the one dimensional case, the unitary convolution \eqref{unit_convo_sev_var} preserves the multiplicativity of
functions. See our paper \cite{Tot2014}, which is a survey on (multiplicative) arithmetic functions of several variables.

\section{Main results}

We first prove the following general result.

\begin{theorem} \label{Th_f_gen_modified} Let $f:\N^k \to \C$ be an arithmetic function \textup{($k\in \N$)}.
Assume that
\begin{equation} \label{cond_several_var}
\sum_{n_1,\ldots,n_k=1}^{\infty} 2^{\omega(n_1)+\cdots + \omega(n_k)} \frac{|(\mu^*_k\times f)(n_1,\ldots,n_k)|}{n_1\cdots n_k}< \infty.
\end{equation}

Then for every $n_1,\ldots,n_k\in \N$,
\begin{equation} \label{f_Raman_mod_several}
f(n_1,\ldots,n_k) = \sum_{q_1,\ldots,q_k=1}^{\infty} \widetilde{a}_{q_1,\ldots,q_k} \widetilde{c}_{q_1}(n_1)\cdots \widetilde{c}_{q_k}(n_k),
\end{equation}
where
\begin{equation} \label{coeff_mod}
\widetilde{a}_{q_1,\ldots,q_k} = \sum_{\substack{m_1,\ldots,m_k=1\\(m_1,q_1)=1,\ldots,(m_k,q_k)=1}}^{\infty}
\frac{(\mu^*_k \times f)(m_1q_1,\ldots,m_kq_k)}{m_1q_1\cdots m_kq_k},
\end{equation}
the series \eqref{f_Raman_mod_several} being absolutely convergent.
\end{theorem}

\begin{proof} We have for any $n_1,\ldots,n_k\in \N$, by using
property \eqref{sum_modified_Raman_sum},
\begin{equation*}
f(n_1,\ldots,n_k) = \sum_{d_1\parallel n_1, \ldots, d_k\parallel n_k} (\mu^*_k \times f)(d_1,\ldots,d_k)
\end{equation*}
\begin{equation*}
= \sum_{d_1,\ldots,d_k=1}^{\infty} \frac{(\mu^*_k \times f)(d_1,\ldots,d_k)}{d_1\cdots d_k}
\sum_{q_1\parallel d_1} \widetilde{c}_{q_1}(n_1) \cdots \sum_{q_k\parallel d_k} \widetilde{c}_{q_k}(n_k)
\end{equation*}
\begin{equation*}
= \sum_{q_1,\ldots,q_k=1}^{\infty} \widetilde{c}_{q_1}(n_1) \cdots \widetilde{c}_{q_k}(n_k) \sum_{\substack{d_1,\ldots,d_k=1\\
q_1\parallel d_1,\ldots,q_k\parallel d_k}}^{\infty} \frac{(\mu^*_k \times f)(d_1,\ldots,d_k)}{d_1\cdots d_k},
\end{equation*}
leading to expansion \eqref{f_Raman_mod_several} with the coefficients \eqref{coeff_mod}, by
denoting $d_1=m_1q_1,\ldots,d_k=m_kq_k$. The rearranging of the terms is justified by the absolute convergence
of the multiple series, shown hereinafter:
\begin{equation*}
\sum_{q_1,\ldots,q_k=1}^{\infty} |\widetilde{a}_{q_1,\ldots,q_k}| |\widetilde{c}_{q_1}(n_1)| \cdots |\widetilde{c}_{q_k}(n_k)|
\end{equation*}
\begin{equation*}
\le \sum_{\substack{q_1,\ldots,q_k=1\\m_1,\ldots,m_k=1\\(m_1,q_1)=1,\ldots,(m_k,q_k)=1}}^{\infty}
\frac{|(\mu^*_k \times f)(m_1q_1,\ldots,m_kq_k)|}{m_1q_1\cdots m_kq_k} |\widetilde{c}_{q_1}(n_1)| \cdots |\widetilde{c}_{q_k}(n_k)|
\end{equation*}
\begin{equation*}
= \sum_{t_1,\ldots,t_k=1}^{\infty} \frac{|(\mu^*_k \times f)(t_1,\ldots,t_k)|}{t_1\cdots t_k}
\sum_{\substack{m_1q_1=t_1\\ (m_1,q_1)=1}} |\widetilde{c}_{q_1}(n_1)| \cdots  \sum_{\substack{m_kq_k=t_k\\ (m_k,q_k)=1}} |\widetilde{c}_{q_k}(n_k)|
\end{equation*}
\begin{equation*}
\le n_1\cdots n_k \sum_{t_1,\ldots,t_k=1}^{\infty} 2^{\omega(t_1)+\cdots +\omega(t_k)} \frac{|(\mu^*_k \times f)(t_1,\ldots,t_k)|}{t_1\cdots t_k}
<\infty,
\end{equation*}
by using inequality \eqref{mod_Raman_abs_ineq} and condition \eqref{cond_several_var}.
\end{proof}

Next we consider the case $f(n_1,\ldots,n_k)= g((n_1,\ldots,n_k)_{*k})$. The following result is the analogue of Theorem \ref{Th_gen}.

\begin{theorem} \label{Th_f_g} Let $g:\N \to \C$ be an arithmetic function and let $k\in \N$. Assume that
\begin{equation*}
\sum_{n=1}^{\infty} 2^{k\, \omega(n)} \frac{|(\mu^* \times g)(n)|}{n^k} < \infty.
\end{equation*}

Then for every $n_1,\ldots,n_k\in \N$,
\begin{equation*}
g((n_1,\ldots,n_k)_{*k}) = \sum_{q_1,\ldots,q_k=1}^{\infty} \widetilde{a}_{q_1,\ldots,q_k} \widetilde{c}_{q_1}(n_1) \cdots \widetilde{c}_{q_k}(n_k),
\end{equation*}
is absolutely convergent, where
\begin{equation} \label{coeff_tilde_a}
\widetilde{a}_{q_1,\ldots,q_k} = \frac1{Q^k} \sum_{\substack{m=1\\ (m,Q)=1}}^{\infty} \frac{(\mu^* \times g)(mQ)}{m^k},
\end{equation}
with the notation $Q=[q_1,\ldots,q_k]$.
\end{theorem}

\begin{proof} We apply Theorem \ref{Th_f_gen_modified}. Taking into account the identity
\begin{equation*}
g((n_1,\ldots,n_k)_{*k})= \sum_{d\parallel n_1,\ldots,d\parallel n_k} (\mu^* \times g)(d)
\end{equation*}
we see that now
\begin{equation*}
(\mu^*_k \times f)(n_1,\ldots,n_k)= \begin{cases} (\mu^* \times g)(n), & \text{ if $n_1=\cdots = n_k=n$,}\\ 0, & \text{ otherwise.}
\end{cases}
\end{equation*}

Therefore the coefficients of the expansion are
\begin{equation*}
\widetilde{a}_{q_1,\ldots,q_k} = \sum_{\substack{n=1\\ m_1q_1=\cdots = m_kq_k=n\\(m_1,q_1)=1,\ldots,(m_k,q_k)=1}}^{\infty}
\frac{(\mu^*_k \times f)(m_1q_1,\ldots,m_kq_k)}{m_1q_1\cdots m_kq_k}
\end{equation*}
\begin{equation*}
=  \sum_{\substack{n=1\\ q_1\parallel n, \ldots, q_k\parallel n}}^{\infty}
\frac{(\mu^* \times g)(n)}{n^k},
\end{equation*}
and we use that $q_1\parallel n,\ldots,q_k\parallel n$ holds if and only if $[q_1,\ldots,q_k]=Q\parallel n$, that is,
$n=mQ$ with $(m,Q)=1$.
\end{proof}

\begin{cor} \label{Cor_sigma_s_star} For every $n_1,\ldots,n_k\in \N$ the following series are absolutely convergent:
\begin{equation} \label{sigma_s_mod}
\frac{\sigma^*_s((n_1,\ldots,n_k)_{*k})}{(n_1,\ldots,n_k)_{*k}^s} = \zeta(s+k) \sum_{q_1,\ldots,q_k=1}^{\infty}
\frac{\phi_{s+k}(Q) \widetilde{c}_{q_1}(n_1)\cdots \widetilde{c}_{q_k}(n_k)}{Q^{2(s+k)}} \quad (s\in \R, s+k>1),
\end{equation}
\begin{equation} \label{tau_star}
\tau^*((n_1,\ldots,n_k)_{*k}) = \zeta(k) \sum_{q_1,\ldots,q_k=1}^{\infty}
\frac{\phi_k(Q) \widetilde{c}_{q_1}(n_1)\cdots \widetilde{c}_{q_k}(n_k)}{Q^{2k}} \quad (k\ge 2).
\end{equation}
\end{cor}

\begin{proof} Apply Theorem \ref{Th_f_g} to $g(n)=\sigma^*_s(n)/n^s$. Here
\begin{equation*}
\mu^* \times g =\mu^* \times \frac{\1 \times \id_s}{\id_s} = (\mu^* \times \1) \times \frac{\1}{\id_s} = \frac{\1}{\id_s},
\end{equation*}
hence $(\mu^* \times g)(n) =1/n^s$ ($n\in \N$). We deduce by \eqref{coeff_tilde_a} that
\begin{equation*}
\widetilde{a}_{q_1,\ldots,q_k} = \frac1{Q^{s+k}} \sum_{\substack{m=1\\ (m,Q)=1}}^{\infty} \frac1{m^{s+k}}
= \zeta(s+k) \frac{\phi_{s+k}(Q)}{Q^{2(s+k)}},
\end{equation*}
which completes the proof.
\end{proof}

In the case $s=1$ identity \eqref{sigma_s_mod} reduces to \eqref{sigma_k_modified}. Now let consider the function $\phi^*_s(n) =
\prod_{p^\nu \parallel n} (p^{s\nu}-1)$, representing the unitary Jordan function of order $s$. Here $\phi^*_s= \mu^* \times \id_s$, and
$\phi^*_1=\phi^*$ is the unitary Euler function, already mentioned in Section \ref{Sect_Funct_unit}.

\begin{cor} \label{Cor_phi_s_star} For every $n_1,\ldots,n_k\in \N$ the following series are absolutely convergent:
\begin{equation} \label{mod_phi_s_star}
\frac{\phi^*_s((n_1,\ldots,n_k)_{*k})}{(n_1,\ldots,n_k)_{*k}^s} = \zeta(s+k)\prod_{p\in \P} \left( 1-\frac{2}{p^{s+k}}\right) \times
\end{equation}
\begin{equation*}
\times \sum_{q_1,\ldots,q_k=1}^{\infty}
\frac{\mu^*(Q) \phi_{s+k}(Q) \widetilde{c}_{q_1}(n_1)\cdots \widetilde{c}_{q_k}(n_k)}{Q^{2(s+k)}\prod_{p\mid Q} (1-2/p^{s+k})}
\quad (s\in \R, s+k>1),
\end{equation*}
\begin{equation*}
\frac{\phi^*((n_1,\ldots,n_k)_{*k})}{(n_1,\ldots,n_k)_{*k}} = \zeta(k+1)\prod_{p\in \P} \left( 1-\frac{2}{p^{k+1}}\right) \times
\end{equation*}
\begin{equation*}
\times \sum_{q_1,\ldots,q_k=1}^{\infty}
\frac{\mu^*(Q) \phi_{k+1}(Q)\widetilde{c}_{q_1}(n_1)\cdots \widetilde{c}_{q_k}(n_k)}{Q^{2(k+1)}\prod_{p\mid Q} (1-2/p^{k+1})}
\quad (k\ge 1).
\end{equation*}
\end{cor}

\begin{proof} Apply Theorem \ref{Th_f_g} to $g(n)=\phi^*_s(n)/n^s$. Here
\begin{equation*}
\mu^* \times g =\mu^* \times \frac{\mu^* \times \id_s}{\id_s} = (\mu^* \times \1) \times \frac{\mu^*}{\id_s} = \frac{\mu^*}{\id_s},
\end{equation*}
that is, $(\mu^* \times g)(n) =\mu^*(n)/n^s$ ($n\in \N$). We deduce by \eqref{coeff_tilde_a} that
\begin{equation*}
\widetilde{a}_{q_1,\ldots,q_k} = \frac1{Q^{s+k}} \sum_{\substack{m=1\\ (m,Q)=1}}^{\infty} \frac{\mu^*(mQ)}{m^{s+k}}
= \frac{\mu^*(Q)}{Q^{s+k}} \sum_{\substack{m=1\\ (m,Q)=1}}^{\infty} \frac{\mu^*(m)}{m^{s+k}}
\end{equation*}
\begin{equation*}
= \frac{\mu^*(Q)}{Q^{s+k}} \zeta(s+k) \prod_{p\in \P} \left(1-\frac{2}{p^{s+k}}\right) \prod_{p\mid Q} \left(1-\frac{1}{p^{s+k}}\right)  \left(1-\frac{2}{p^{s+k}}\right)^{-1},
\end{equation*}
leading to \eqref{mod_phi_s_star}.
\end{proof}

For $m\in \N$, $m\ge 2$ consider the function $g(n)=m^{\omega(n)}$, which is the unitary analogue of the Piltz divisor function $\tau_m(n)$.
Here $m^{\omega(n)}=\sum_{d\parallel n} (m-1)^{\omega(d)}$ for any $n\in \N$. We obtain by similar arguments:

\begin{cor} \label{Cor_unit_Piltz} For every $n_1,\ldots,n_k\in \N$ the following series is absolutely convergent:
\begin{equation} \label{Piltz_unit}
m^{\omega((n_1,\ldots,n_k)_{*k})} = \zeta(k) \prod_{p\in \P} \left( 1+\frac{m-2}{p^k}\right) \times
\end{equation}
\begin{equation*}
\times \sum_{q_1,\ldots,q_k=1}^{\infty} \frac{\phi_k(Q)(m-1)^{\omega(Q)} \widetilde{c}_{q_1}(n_1)\cdots
\widetilde{c}_{q_k}(n_k)}{Q^{2k} \prod_{p\mid Q} (1+(m-2)/p^k)} \quad (m,k\ge 2),
\end{equation*}
\end{cor}

For $m=2$ identity \eqref{Piltz_unit} reduces to \eqref{tau_star}.

\begin{remark} \textup{It is possible to formulate the results of Theorem \ref{Th_f_gen_modified} in the case of multiplicative
functions $f$ of $k$ variables, and Theorem \ref{Th_f_g} in the case of multiplicative functions $g$ of one variable. Note
that if $g$ is multiplicative, then
$f(n_1,\ldots,n_k)= g((n_1,\ldots,n_k)_{*k})$ is multiplicative, viewed as a function of $k$ variables. See also Delange \cite{Del1976}
and the author \cite{Tot2017}. Furthermore, it is possible to apply the above results to other special (multiplicative)
functions. We do not go into more details.}
\end{remark}

\section{Acknowledgement} Supported by the European Union, co-financed by the European Social Fund
EFOP-3.6.1.-16-2016-00004.


\begin{thebibliography}{99}

\bibitem{Coh1960} E.~Cohen, Arithmetical functions associated with the unitary divisors of an integer,
{\it Math. Z.} {\bf 74} (1960), 66--80.

\bibitem{Del1976} H.~Delange, On Ramanujan expansions of certain arithmetical functions, {\it Acta. Arith.} {\bf 31}
(1976), 259--270.

\bibitem{Gry1981} A.~Grytczuk, An identity involving Ramanujan's sum. {\it Elem. Math.} {\bf 36} (1981), 16--17.

\bibitem{Klu1906} J.~C.~Kluyver, Some formulae concerning the integers less than $n$ and prime to $n$, In: {\it Proceedings of
the Royal Netherlands Academy of Arts and Sciences (KNAW)} vol. 9, pp. 408--414,  1906.

\bibitem{Luc2010} L.~G.~Lucht, A survey of Ramanujan expansions, {\it Int. J. Number Theory} {\bf 6} (2010), 1785--1799.

\bibitem{McC1986} P.~J.~McCarthy, {\it Introduction to Arithmetical
Functions}, Springer Verlag, New York - Berlin - Heidelberg - Tokyo, 1986.

\bibitem{Ramanujan1918} S.~Ramanujan, On certain trigonometric sums and their applications in the theory of numbers,
{\it Trans. Cambridge Philos. Soc.} {\bf 22} (1918), 259--276; Collected Papers 179--199.

\bibitem{Ram2013} M.~Ram~Murty, Ramanujan series for arithmetic functions,
Hardy-Ramanujan J. {\bf 36} (2013), 21--33.

\bibitem{SchSpi1994} W.~Schwarz, J.~Spilker, {\it Arithmetical
functions, An introduction to elementary and analytic properties of
arithmetic functions and to some of their almost-periodic
properties}, London Mathematical Society Lecture Note Series, 184.
Cambridge University Press, Cambridge, 1994.

\bibitem{SitSur1973} R.~Sitaramachandrarao and D.~Suryanarayana, On $\sum_{n\le x} \sigma^*(n)$ and
$\sum_{n\le x} \varphi^*(n)$,  {\it Proc. Amer. Math. Soc.} {\bf 41} (1973), 61--66.

\bibitem{SitSub2007} V.~Sitaramaiah and M.~V.~Subbarao, Unitary analogues of some formulae of Ingham,
{\it Ars Combin.} {\bf 84} (2007), 33--49.

\bibitem{Sur1970} D.~Suryanarayana, A property of the unitary analogue of Ramanujan's sum, {\it
Elem. Math.} {\bf 25} (1970), 114.

\bibitem{Tot2009} L.~T\'oth, On the bi-unitary analogues of Euler's arithmetical function and the gcd-sum function,
{\it J. Integer Seq.} {\bf 12} (2009), Article 09.5.2, 10 pp.

\bibitem{Tot2014} L.~T\'oth, Multiplicative arithmetic functions of several variables:
a survey, in {\it Mathematics Without Boundaries, Surveys in Pure
Mathematics}. Th.~M.~Rassias, P.~Pardalos (Eds.), Springer, New
York, 2014, 483--514. \url{arXiv:1310.7053 [math.NT]}

\bibitem{Tot2017} L.~T\'oth, Ramanujan expansions of arithmetic functions of several variables, {\it Ramanujan J.}, accepted.
\url{arXiv:1704.02881 [math.NT]}

\bibitem{Tru2015} T.~Trudgian, The sum of the unitary divisor function, {\it Publ. Inst. Math. (Beograd) (N.S.)} {\bf 97 (111)}
(2015), 175--180.

\bibitem{Vai1931} R.~Vaidyanathaswamy, The theory of multiplicative arithmetic
functions, {\it Trans. Amer. Math. Soc.}, {\bf 33} (1931), 579--662.

\bibitem{Ush2016} N.~Ushiroya, Ramanujan-Fourier series of certain arithmetic
functions of two variables, {\it Hardy-Ramanujan J.} {\bf 39} (2016), 1--20.

\end{thebibliography}
\end{document}